\documentclass[final,1p,times,authoryear]{elsarticle}

\usepackage{amsmath}
\usepackage{amsthm}
\usepackage{amssymb}
\usepackage{amsfonts}
\usepackage{paralist}
\usepackage{aliascnt}
\usepackage[colorlinks=true,linkcolor=green,urlcolor=green,hyperfootnotes=true]{hyperref}
\usepackage[title]{appendix}
\usepackage{tikz}
\usetikzlibrary{matrix}

\newcommand{\RR}{\mathbb{F}}
\newcommand{\kk}{\mathbb{K}}
\newcommand{\NN}{\mathbb{N}}

\newcommand{\xx}{\mathbf{x}}
\newcommand{\mm}{\mathfrak{m}}
\newcommand{\rank}{\normalfont\text{rank}}
\newcommand{\reg}{\normalfont\text{reg}}

\newcommand{\grade}{\normalfont\text{grade}}
\newcommand{\Tor}{\normalfont\text{Tor}}

\newcommand{\pd}{\normalfont\text{pd}}
\newcommand{\Ker}{\normalfont\text{Ker}}
\newcommand{\im}{\normalfont\text{Im}}
\newcommand{\HT}{\normalfont\text{ht}}
\newcommand{\Syz}{\normalfont\text{Syz}}
\newcommand{\II}{\hat{I}}
\newcommand{\pp}{\normalfont\mathfrak{p}}
\newcommand{\Spec}{\normalfont\text{Spec}}

\newtheorem{headthm}{Theorem}

\newtheorem{theorem}{Theorem}
\newaliascnt{lemma}{theorem}
\newtheorem{lemma}[lemma]{Lemma}
\aliascntresetthe{lemma}

\newaliascnt{corollary}{theorem}
\newtheorem{corollary}[corollary]{Corollary}
\aliascntresetthe{corollary}

\newaliascnt{conjecture}{theorem}

\aliascntresetthe{conjecture}

\newaliascnt{proposition}{theorem}
\newtheorem{proposition}[proposition]{Proposition}
\aliascntresetthe{proposition}

\newaliascnt{definition}{theorem}
\newtheorem{definition}[definition]{Definition}
\aliascntresetthe{definition}

\newtheorem*{notation}{Notation}

\newaliascnt{example}{theorem}
\newtheorem{example}[example]{Example}
\aliascntresetthe{example}

\newaliascnt{remark}{theorem}
\newtheorem{remark}[remark]{Remark}
\aliascntresetthe{remark}

\newaliascnt{problem}{theorem}
\newtheorem{problem}[problem]{Problem}
\aliascntresetthe{problem}

\newaliascnt{construction}{theorem}
\newtheorem{construction}[construction]{Construction}
\aliascntresetthe{construction}

\newaliascnt{question}{theorem}

\aliascntresetthe{question}

\def\equationautorefname~#1\null{(#1)\null}
\def\sectionautorefname~#1\null{Appendix\null}
\def\sectionautorefname~#1\null{Section #1\null}

\begin{document}
	
	\begin{frontmatter}

		\title{Bounding the degrees of a minimal $\mu$-basis for a rational surface parametrization.}
		
		\author{Yairon Cid-Ruiz}
		
		\address{Department de Matem\`{a}tiques i Inform\`{a}tica, Facultat de Matem\`{a}tiques i Inform\`{a}tica, Universitat de Barcelona, Gran Via de les Corts Catalanes, 585; 08007 Barcelona, Spain.}
		
		\ead{ycid@ub.edu}
		\ead[url]{http://www.ub.edu/arcades/ycid.html}
		
		\begin{abstract}
			In this paper,  we study how the degrees of the elements in a minimal $\mu$-basis of a parametrized surface behave. 			
			For an arbitrary rational surface parametrization $P(s,t)=(a_1(s,t),a_2(s,t),a_3(s,t),\\a_4(s,t)) \in \RR[s,t]^4$ over an infinite field $\RR$, we show the existence of a $\mu$-basis with polynomials bounded in degree by $O(d^{33})$, where $d=\max(\deg(a_1),\deg(a_2), \deg(a_3), \deg(a_4))$.
			Under additional assumptions we can obtain tighter bounds.
			
		\end{abstract}
		
		
	\end{frontmatter}

	\section{Introduction}

	The concept of a $\mu$-basis is an important notion in Computer Aided Geometric Design and Geometric Modeling, that was introduced in \cite{MU_BASIS_CURVES} to study the implicitization problem in the case of parametrized curves.
	The $\mu$-basis of a parametric curve is a well-understood object that provides the implicit equation by computing its resultant \cite[Section 4]{MU_BASIS_CURVES} and has a number of applications in the study of rational curves (see e.g.  \cite{CHEN_SEDERB_IMP, CHEN_WANG_YANG_SING, GOLDMAN_JIA}).
	On the other hand, the $\mu$-basis of a parametric surface is a more complicated object and its development took several years of research. 
	In a first attempt, for the particular case of rational ruled surfaces the concept of a $\mu$-basis was defined in  \cite{MU_BASIS_RULED_SURF} and \cite{CHEN_WANG_REVIS} (also, see \cite{DOHM}).
	Later, in \cite{MU_BASIS}, the existence of a $\mu$-basis was proved for an arbitrary rational surface.

	\medskip	
	
	The existence of $\mu$-bases for rational surfaces is a strong result whose  geometrical meaning is that any rational surface is the intersection of three moving planes without extraneous factors.
	Additionally, a $\mu$-basis of a rational surface parametrization coincides with a basis of the syzygy module \cite[Corollary 3.1]{MU_BASIS} and can be used to obtain the implicit equation \cite[Section 4]{MU_BASIS}.
	Contrary to the case of rational curves, there is no known upper bound for the degrees of the elements in a minimal $\mu$-basis of a rational surface parametrization. 
	The main purpose of this article is to obtain such an upper bound. 
	
	\medskip
	
	In order to describe the main results of this paper, we briefly recall the notion of a $\mu$-basis for a rational parametric surface.

	Let $\RR$ be an infinite field and $R$ be the polynomial ring $R=\RR[s,t]$. 
	
	\begin{definition}
		A rational surface parametrization in homogeneous form is defined by
		\begin{equation}
			\label{rational_surface}
			P(s, t) = \big(a_1(s, t), a_2(s, t), a_3(s, t), a_4(s, t)\big)
		\end{equation}
		where $a_1, a_2, a_3, a_4 \in R$ and $\gcd(a_1, a_2, a_3, a_4) = 1$. 
	\end{definition} 
	
	\begin{definition}	
		A moving plane following the rational parametrization \autoref{rational_surface} is a quadruple 
		$$
		\big(A(s, t), B(s, t), C(s, t), D(s, t)\big) \in R^4
		$$
		such that 
		\begin{equation}
			\label{moving_plane}
			A(s, t)a_1(s,t) + 
			B(s, t)a_2(s,t) +  C(s, t)a_3(s,t) + D(s, t)a_4(s,t) = 0.
		\end{equation}
	\end{definition}
	
	\begin{definition}
		Let $\mathbf{p}=(p_1,p_2,p_3,p_4), \mathbf{q}=(q_1,q_2,q_3,q_4), \mathbf{r}=(r_1,r_2,r_3,r_4)$ be three moving planes such that 
		\begin{equation}
			\label{equation_MU_BASIS}
			[\mathbf{p},\mathbf{q},\mathbf{r}]=\alpha P(s,t)
		\end{equation}
		for some nonzero constant $\alpha \in \RR$. Then $\mathbf{p},\mathbf{q},\mathbf{r}$ are said to be a $\mathbf{\mu-basis}$ of the rational surface parametrization \autoref{rational_surface}. 
		Here $[\mathbf{p},\mathbf{q},\mathbf{r}]$ is defined as the outer product
		$$
		[\mathbf{p},\mathbf{q},\mathbf{r}]=\left(
		\left|
		\begin{array}{ccc}
		p_2 & p_3 & p_4 \\
		q_2 & q_3 & q_4 \\
		r_2 & r_3 & r_4 
		\end{array}
		\right|
		,-
		\left|
		\begin{array}{ccc}
		p_1 & p_3 & p_4 \\
		q_1 & q_3 & q_4 \\
		r_1 & r_3 & r_4 
		\end{array}
		\right|,
		\left|
		\begin{array}{ccc}
		p_1 & p_2 & p_4 \\
		q_1 & q_2 & q_4 \\
		r_1 & r_2 & r_4 
		\end{array}
		\right|
		,-
		\left|
		\begin{array}{ccc}
		p_1 & p_2 & p_3 \\
		q_1 & q_2 & q_3 \\
		r_1 & r_2 & r_3 
		\end{array}
		\right|		
		\right).
		$$	
	\end{definition} 
	
	As noted before, there is a $1-1$ relation between a $\mu$-basis and a basis for the syzygy module $\Syz(a_1,a_2,a_3,a_4)$ given in the following form:
	\begin{compactitem}
		\item \cite[Corollary 3.1]{MU_BASIS} $\mathbf{p}$, $\mathbf{q}$ and $\mathbf{r}$ form a $\mu$-basis if and only if $\mathbf{p}$, $\mathbf{q}$ and $\mathbf{r}$ are a basis of $\Syz(a_1, a_2, a_3, a_4)$.
	\end{compactitem}

	For any vector $\mathbf{v} \in R^m$ we denote its degree by $\deg(\mathbf{v})=\max_{j}\{\deg(\mathbf{v}_j)\}$, where $\deg(\mathbf{v}_j)$ is equal to the total degree of the polynomial $\mathbf{v}_j$ in the variables $s$ and $t$.
 	
	\begin{definition}
		$\mathbf{p}$, $\mathbf{q}$ and $\mathbf{r}$ are said to form a \textbf{minimal $\mathbf{\mu}$-basis} of the rational surface \autoref{rational_surface} if  among all the triples satisfying \autoref{equation_MU_BASIS}, $\deg(\mathbf{p})+\deg(\mathbf{q})+\deg(\mathbf{r})$ is the smallest.
	\end{definition}
	
	In \cite{MU_BASIS} there are left several questions of interest for further research and better understanding. Here we try to address the question:
	\begin{compactitem}
		\item \textbf{What can be said about the degrees of the polynomials in a minimal $\mathbf{\mu}$-basis?}
	\end{compactitem}
	that was asked in \cite[Section 5, second question]{MU_BASIS}.
	
	Due to the equivalence between being a $\mu$-basis and being a basis for the syzygy module, this question is the same as finding an upper bound for the latter one.
	We remark that the problem of studying the degrees of the syzygies of an ideal or a module has attracted the attention of several researchers (see e.g. \cite{LAZARD_ALG_LIN, LAZARD_UPPER_BOUND_NOTE, SYZ_COMP,YAP_LOWER, PEEVA_STURM,AVARAM_LUCH_CONC}).
		
	Another interesting feature of the $\mu$-bases is that they form the linear part of the moving curve/surface ideals in the case of curves/surfaces. 
	In \cite{COX_MOVING}, it was noticed that computing the moving curve/surface ideal is the same as determining the defining equations of the Rees algebra of the ideal generated by the parametrization of the curve/surface.
	The problem of finding the presentation of the Rees algebra is a long standing problem in commutative algebra and algebraic geometry that is a very active research topic (see e.g. \cite{VASC_EQ_REES, COX_REES_MU_1, HONG_SIMIS_VASC_ELIM, LAURENT, CARLOS_MONOID, CARLOS_MU2, CARLOS_MONO,  KPU_NORMAL_SCROLL, KPU_GOR3, DMOD, SAT_SPEC_FIB}).

	\medskip

	In the following construction we homogenize the ideal $I = (a_1,a_2,a_3,a_4) \subset R$ defined by the parametrization \autoref{rational_surface}.

	\begin{construction}
		Given the data $\{a_1,a_2,a_3,a_4\}$ that determines \autoref{rational_surface}, then we define the homogeneous ideal $\II=(b_1,b_2,b_3,b_4)$ with generators
		\begin{equation}
		\label{homogenization}
		b_i(s,t,u) = u^d a_i(\frac{s}{u}, \frac{t}{u}) \in \RR[s,t,u],
		\end{equation}	
		where $d=\max(\deg(a_1), \deg(a_2), \deg(a_3), \deg(a_4))$.
	\end{construction}

	The main result of this paper is the following theorem where we find upper bounds for the degrees of the elements in a minimal $\mu$-basis.
	
	\begin{headthm}[\autoref{result_pd_2}]
		\label{Main_Thm}
		Let $P(s,t)$ be the parametrization in \autoref{rational_surface} and $d$ be the number
		 $$
		 d=\max\{\deg(a_1),\deg(a_2),\deg(a_3),\deg(a_4)\}.
		 $$
		 Then, the following statements hold:
		 \begin{enumerate}[(i)]
			\item There exists a $\mu$-basis with polynomials bounded in the order of $O(d^{33})$.
			\item If the homogenized ideal $\II$ (obtained in \autoref{homogenization}) has height $\HT(\II)=3$, then there exists a $\mu$-basis with degree bounded by $O(d^{22})$.
			\item If the homogenized ideal $\II$ is a ``general'' Artinian almost complete intersection (i.e, like in \autoref{general_form}), then there exists a $\mu$-basis with degree bounded by $O(d^{12})$.
			\item If the homogenized ideal $\II$ has projective dimension $\pd(\II)=1$, then there exists a $\mu$-basis with degree bounded by $d$.
		 \end{enumerate}
	\end{headthm}
	
	The proof of  \autoref{Main_Thm} is based on two fundamental ingredients. 
	By using techniques coming from homological and commutative algebra we bound numerical invariants of the minimal free resolution (e.g. regularity and Betti numbers) of the ideal $\II$ obtained by homogenizing the ideal $I=(a_1,a_2,a_3,a_4)$, and then a process of dehomogenization gives us a presentation of $\Syz(a_1,a_2,a_3,a_4)$ where everything can be bounded in terms of $d$.
	Under the assumptions of working over an infinite field $\RR$ and having a presentation of $\Syz(a_1,a_2,a_3,a_4)$, then we apply the remarkable results     of \cite{SERRE_ART} where an effective version of the Quillen-Suslin Theorem is given.
	
	In the part $(iii)$ we use an explicit description of the minimal free resolution of a general Artinian almost complete intersection, that was obtained in \cite{ROIG}. 
	The part $(iv)$ follows from \cite{STRONG_MU_BASIS} where the case $\pd(\II)=1$ was studied and called \textit{strong $\mu$-basis}.
	
	In contrast to our results, the elements of a $\mu$-basis of a parametric rational curve of degree $d$ are bounded in degree by exactly $d$.
	This big difference between the case of curves and surfaces comes from the fact that the syzygy module of the homogenized ideal may not be free in the case of surfaces but in the case of curves is always free.
	Actually, the condition of $(iv)$ accounts to say that the syzygy module of $\II$ is free, and the case of a parametric rational surface having a strong $\mu$-basis is treated similarly to the case of rational curves.
	In the general case where the syzygy module of $\II$ is not free, then the dehomogenization process that we use does not give us a basis of $\Syz(a_1,a_2,a_3,a_4)$. 
	To overcome this difficulty, we use the effective version of Quillen-Suslin Theorem in \cite{SERRE_ART}, and it is in this last step  where the complexity of our upper bounds becomes large.
	As a general opinion, we think that our upper bounds are not sharp.
	
	\medskip
	
	In our proof of \autoref{Main_Thm} we needed to find some upper bounds for the regularity and Betti numbers of the homogeneous ideal $\II$.
	Since we think that these auxiliary upper bounds may be of interest on their own, we worked with more general ideals and obtained the following results:
	\begin{enumerate}[(i)]
		\item Let $\kk$ be an arbitrary field. For a homogeneous ideal $J=(f_1,f_2,\ldots,f_m) \subset \kk[s,t,u]$ generated by $m\ge 2$ relatively prime polynomials in $\kk[s,t,u]$, in \autoref{bounds} we give upper bounds for the regularity and the Betti numbers of $J$.  
		\item For a homogeneous ideal $J=(g_1,g_2,g_3,g_4) \subset \RR[s,t,u]$ with $\deg(g_1)=\cdots=\deg(g_4)$ and $\HT(I)=3$, in \autoref{HH3} we improve the upper bounds for the Betti numbers of $J$.
	\end{enumerate}

\medskip

The  basic outline of this paper is as follows.
In \autoref{section2}, we study the syzygies of ideals in a polynomial ring, and in particular we show that $\Syz(a_1,a_2,a_3,a_4)$ is a free module of rank $3$.
In \autoref{section3}, we compute upper bounds for the regularity and the Betti numbers of ideals generated by relatively prime polynomials in three variables.
In \autoref{section4}, by applying the effective version of Quillen-Suslin Theorem in \cite{SERRE_ART}, we prove \autoref{Main_Thm}.
In \autoref{section5}, we briefly discuss the sharpness of our upper bounds.
In \autoref{section6}, we give a simple example to show the process of computing $\mu$-bases with our method.

\medskip

Finally, for the sake of completeness we recall some basic definitions that will be used.
For notational purposes, let $M$ be an $R$-module and $J\subset R$ be an ideal.
The \textit{projective dimension} of $M$, denoted by $\pd(M)$, is the smallest possible length of a projective resolution of $M$ (see \cite[page 233]{ROTMAN}).
The \textit{height} of $J$, denoted by $\HT(J)$, is equal to $\HT(J)=\inf\{\HT(\pp) \mid  I \subset \pp \in \Spec(R) \}$, where the height of a prime ideal $\pp$ is the maximum of the lengths of increasing chains of prime ideals contained in $\pp$  (see \cite[Section 5]{MATSUMURA}).
The \textit{grade} of $J$, denoted by $\grade(J)$, is the maximum of the lengths of the regular sequences contained in $J$ (see \cite[Definition 1.2.6]{BRUNS_HERZ}).

Assume in addition that $M$ is a finitely generated graded $R$-module.
The $k$-th graded component of $M$ is denoted by $M_k$.
The \textit{Hilbert function} of $M$, denoted by $H_M(k)$, is equal to $\dim_{\RR}(M_k)$.
The \textit{minimal free resolution} of $M$ is unique up to isomorphism (see \cite[Theorem 7.5]{GRADED}), then as a consequence, we can define the \textit{Betti numbers} of $M$ (see \cite[Section 11]{GRADED}) and the \textit{regularity} of $M$ (see \cite[Section 18]{GRADED}).
	
	\section{Dealing with syzygies}\label{section2}
	
	The fact that $\Syz(a_1,a_2,a_3,a_4)$ is a free module of rank $3$ is an important step in \cite{MU_BASIS} to show the existence of a $\mu$-basis. In this section we give a different proof for that statement, which we also generalize because we will need the case of three variables after homogenizing the ideal $I=(a_1,a_2,a_3,a_4)$. 
	
	In this section we use the following notation.
	
	\begin{notation}
		Let $\kk$ be an arbitrary field and $R$ be the polynomial ring $R=\kk[x_1,\ldots,x_n]$ where $n\ge 2$.
	\end{notation}
	
	\begin{theorem}
		\label{free_syzygy}
		Let $I$ be an ideal in $R$. Then, for any projective resolution 
		$$\cdots \xrightarrow{d_{n}} P_{n-1} \xrightarrow{d_{n-1}} P_{n-2} \xrightarrow{d_{n-2}} P_{n-3} \xrightarrow{d_{n-3}} \cdots \xrightarrow{d_1} P_0 \xrightarrow{d_0} I \rightarrow 0
		$$
		where the $P_i$'s are finitely generated, the corresponding \textit{(n-2)-th} syzygy $K_{n-2}=\Ker(d_{n-2})$ is free.
	\end{theorem}	
	\begin{proof}	
		By the Hilbert Syzygy Theorem \cite[Corollary 9.36]{ROTMAN}, there exists a finite free resolution of length at most $n$ for the quotient ring $R/I$. 
		We assume that it has length $n$ and we denote it by 
		$$
		0 \rightarrow F_n \rightarrow F_{n-1} \rightarrow \cdots \rightarrow F_0 \rightarrow R/I \rightarrow 0,
		$$
		because if it has length smaller than $n$ then we can simply fill it with zero modules.
		
		With the given projective resolution of $I$ we get the exact sequence 
		$$
		0 \rightarrow K_{n-2} \rightarrow P_{n-2} \xrightarrow{d_{n-2}} P_{n-3} \xrightarrow{d_{n-3}} \cdots \xrightarrow{d_1} P_0 \xrightarrow{d_0} R \rightarrow R/I \rightarrow 0,
		$$
		where $K_{n-2}=\Ker(d_{n-2})$ is the \textit{(n-2)-th} syzygy. Then, from the generalized Schanuel Lemma (see \cite[Theorem 189]{KAP})  we have the isomorphism 
		$$
		K_{n-2} \,\bigoplus\,
 \left(\bigoplus_{j=0}^{\lfloor\frac{n-1}{2}\rfloor} F_{n-1-2j} \right)\,\bigoplus\, \left(\bigoplus_{j=0}^{\lfloor\frac{n-3}{2}\rfloor} P_{n-3-2j} \right) \;\cong\; \left(\bigoplus_{j=0}^{\lfloor\frac{n}{2}\rfloor} F_{n-2j} \right) \,\bigoplus\, \left(\bigoplus_{j=0}^{\lfloor\frac{n-2}{2}\rfloor} P_{n-2-2j} \right),
		$$
		which implies that $K_{n-2}$ is a projective module. Since $R$ is Noetherian and $P_{n-2}$ is finitely generated, then $P_{n-2}$ is Noetherian and $K_{n-2} \subset P_{n-2}$ is finitely generated. Finally, the Quillen-Suslin Theorem \cite[Theorem 4.59]{ROTMAN} implies that the module $K_{n-2}$ is free.
	\end{proof} 
	
	\begin{corollary}
		\label{pd_ineq_ideal}
		For any ideal $I \subset R$ we have $\pd(I) \le n-1$.
		\begin{proof}
			Using that $R$ is Noetherian, for the ideal $I$ we can always find a free resolution composed of finitely generated modules.
			So the corollary follows from \autoref{free_syzygy}.
		\end{proof}
	\end{corollary}
	
	We finish this section by proving that the free module $\Syz(a_1,a_2,a_3,a_4)$ has rank $3$.
	
	\begin{lemma}
		\label{Euler_char}
		Let $A$ be a Noetherian ring, and $I \subset A$ be a nonzero ideal with a finite free resolution. Then $\rank(I)=1$.
		\begin{proof}
			From $0\rightarrow I \rightarrow A \rightarrow A/I \rightarrow 0$ and the additivity of the $\rank$ function \cite[Proposition 1.4.5]{BRUNS_HERZ}, we obtain $\rank(A)=\rank(I)+\rank(A/I)$.
			We always have $\rank(A)=1$ and from  \cite[Theorem 195]{KAP} we get $\rank(A/I)=0$. Therefore $\rank(I)=1$.
		\end{proof}	
	\end{lemma}
	
	From the Hilbert Syzygy Theorem we know that any finitely generated module over $R$ has a finite free resolution, so we are free to apply the previous \autoref{Euler_char} in this case. 
	
	\begin{corollary}
		\label{syzygy_rank}
		Let $I=(f_1,f_2,\ldots,f_m)$ be an ideal in $R$ with $\pd(I)=1$. Then $\Syz(f_1,f_2,\ldots,f_m)$ is a free module of rank $m-1$.
		\begin{proof}
			The hypothesis $\pd(I)=1$ implies that the ideal $I$ has a projective resolution of the form
			$$
			0 \rightarrow P_1 \rightarrow P_0 \rightarrow I \rightarrow 0.
			$$
		
			To study the module $\Syz(f_1,f_2,\ldots,f_m)$, we consider the short exact sequence 
			\begin{equation}
				\label{EQ_SYZ_SHORT}
				0 \rightarrow\Syz(f_1,f_2,\ldots,f_m) \rightarrow R^m \rightarrow I \rightarrow 0.
			\end{equation}
			Then, the Schanuel lemma (see e.g. \cite[page 841, Lemma 2.4]{LANG_ALGEBRA}) yields the following isomorphism 
			$$
			\Syz(f_1,f_2,\ldots,f_m) \oplus P_0 \;\cong\; P_1 \oplus R^m,
			$$
			which implies that $\Syz(f_1,f_2,\ldots,f_m)$ is a projective module. 
			Since $\Syz(f_1,f_2,\ldots,f_m)$ is also finitely generated, then the Quillen-Suslin Theorem \cite[Theorem 4.59]{ROTMAN} implies that $\Syz(f_1,f_2,\ldots,f_m)$ is actually free.
			Finally, the short exact sequence \autoref{EQ_SYZ_SHORT}, the additivity of the rank function and \autoref{Euler_char} give us 
			$$
			\rank\left(\Syz(f_1,f_2,\ldots,f_m)\right)=\rank(R^m) - \rank(I) = m -1,
			$$
			and so the claim of the corollary follows.
		\end{proof}
	\end{corollary}

	\section{Upper bounds for the regularity and the Betti numbers}\label{section3}
	
	This section is devoted to finding an upper bound for the regularity and the Betti numbers of $\II$. 
	Since the results of this section could be of general interest, we will deal with the case of homogeneous ideals generated by relatively prime polynomials in three variables.
	Our main reference will be the chapter \textit{``Graded Free Resolutions''} of \cite{GRADED}.
	
	First we note that the homogenized ideal $\II$ satisfies the condition of being generated by relatively prime polynomials.
	\begin{lemma}
		For the ideal $\hat{I}=(b_1,b_2,b_3,b_4)$ constructed in \autoref{homogenization}
		we have that $b_1,b_2,b_3,b_4$ are relatively prime (i.e. $\gcd(b_1,b_2,b_3,b_4)=1$).
		\begin{proof}
			Suppose that $g = \gcd(b_1,b_2,b_3,b_4) \neq 1\in \RR[s,t,u]$. 
			Since $g(s,t,1) \mid a_i(s,t)$ and $\gcd(a_1,a_2,a_3,a_4)=1$, then we necessarily have that $g \in \kk[u]$.
			By construction one of the $b_i$'s has a term that is free of $u$, without loss of generality we assume that $b_1(s,t,u) = \lambda s^{\alpha}t^{d-\alpha} + u p(s,t,u)$ with $p \in \RR[s,t,u]$ and $\lambda\neq 0$.
			So, since $b_1$ is homogeneous of degree $d$, we have that $g \mid b_1$ is a contradiction.
		\end{proof}
	\end{lemma}
	
	During the present section we use the following notation.
	\begin{notation}
		Let $\kk$ be an arbitrary field and $\RR$ be an infinite field.
		Let $T$ and $S$ be the polynomial rings $T=\kk[s,t,u]$ and $S=\RR[s,t,u]$.
	\end{notation}

	We divide the section into two different parts.
	In the first part, we consider a homogeneous ideal $J=(f_1,f_2,\ldots,f_m) \subset T$
	 generated by $m\ge 2$ relatively prime polynomials.
	In the second part, we deal with the special case of an ideal $J = (g_1,g_2,g_3,g_4) \in S$ with $\deg(g_1)=\cdots=\deg(g_4)$ and $\HT
	(J)=3$.
		
	\begin{theorem}
		\label{bounds}
		Let $m\ge 2$, $J=(f_1,f_2,\ldots,f_m) \subset T$  be a homogeneous ideal, $\gcd(f_1,\ldots,f_m)=1$ and $\deg(f_1),\ldots,\deg(f_m) \le d$. 
		Then, the following statements hold:
		\begin{enumerate}[(i)]
			\item $\reg(J) \le 3d-2.$
			\item 
			$
			\beta_1(J) \le \beta_2(J) + m -1.
			$
			\item 
			$
			\beta_2(J)  \le H_J(\reg(J)) \le H_J(3d-2) \le \binom{3d}{2}.
			$
			
			In addition if $\deg(f_1)=\deg(f_2)=\ldots=\deg(f_m)=d$ then 
			$
			\beta_2(J) \le m \binom{2d}{2}.
			$			
		\end{enumerate} 
	\end{theorem}
	
	We break the proof of \autoref{bounds} in some steps that now follow.
	First, we prove that any ideal as $J$ above has two relatively prime elements, but in order to prove it we have to make a more complicated reformulation.
	
	\begin{lemma}
		\label{lemma_gcd}
		Let $m \ge 2$ and $f_1, \ldots,f_m \in \kk[x_1,\ldots,x_n]$ be relatively prime polynomials (i.e. $\gcd(f_1,\ldots, f_m)=1$). 
		Then there exists an infinite sequence of polynomials $\{h_i\}_{i=1}^{\infty} \subset (f_1,\ldots,f_m)$, with $\gcd(h_i,h_j)=1$ for $i \ne j$. 
		\begin{proof}
			We proceed by an induction argument on $m$.
     		Fix $m\ge 2$.
			We compute $g=\gcd(f_1,\ldots,f_{m-1})$ and the new polynomials $f_1^{'}=f_1/g, \ldots, f_{m-1}^{'}=f_{m-1}/g$. 
			In the case $m=2$ we have $f_1^{'}=1$, and when $m>2$ we get $\gcd(f_1^{'},\ldots,f_{m-1}^{'})=1$.
			Hence, in both cases, we can obtain an infinite sequence $\{h_i^{'}\}_{i=1}^{\infty} \subset (f_1^{'},\ldots,f_{m-1}^{'})$ with $\gcd(h_i^{'}, h_j^{'})=1$ for $i \ne j$, because when $m=2$ we have $(f_1^{'})=\kk[x_1,\ldots,x_n]$ and when $m>2$ we can use the induction hypothesis.
			
			For each $h_i^{'}$ we have that $\gcd(f_m,f_m+gh_i^{'})=\gcd(f_m,gh_i^{'})$. From $\gcd(f_1,\ldots,f_m)=1$ we conclude that $\gcd(f_m,g)=1$, and for some $j \in \NN$ we should have $\gcd(f_m,gh_j^{'})=1$, because all the $h_i^{'}$'s have different prime factors but $f_m$ can have only a finite amount of prime factors.
			
			Suppose we have computed a sequence of polynomials $h_1,\ldots,h_k$ and a polynomial $g_k$, with the properties $\gcd(h_i,h_j)=1$ for $1\le i < j \le k$ and $\gcd(h_i,g_k)=1$ for $1\le i \le k$.
			Again, for each $h_i{'}$ we have  
			\begin{equation}
				\label{gcd_equation}
				\gcd(h_1\cdots h_k, h_1\cdots h_k + g_kh_i^{'})=\gcd(h_1\cdots h_k, g_kh_i^{'})=\gcd(h_1\cdots h_k+g_kh_i^{'},g_kh_i^{'}),	
			\end{equation}
			and there must exist some $j\in\NN$	with $\gcd(h_1h_2\cdots h_k, g_kh_j^{'})=1$. 
			Thus we define the next elements in the inductive step as 
			$
			h_{k+1}=h_1h_2\cdots h_k + g_kh_j^{'}$ and  $g_{k+1}=g_kh_j^{'}.
			$	
			
			From \autoref{gcd_equation} we have that $\gcd(h_1h_2\cdots h_k,h_{k+1})=\gcd(h_1h_2\cdots h_k,g_{k+1})=\gcd(h_{k+1},g_{k+1})=1$, which implies $\gcd(h_i,h_j)=1$ for $1\le i < j \le k+1$ and $\gcd(h_i,g_{k+1})=1$ for $1\le i \le k+1$.
			Starting with $h_1=f_m,\; g_1=g$ and following this iterative process we can 
			construct the required sequence $\{h_i\}_{i=1}^{\infty} \subset (f_1,\ldots,f_m)$ with $\gcd(h_i,h_j)=1$ for $i \neq j$.
		\end{proof}
	\end{lemma}
	
		\begin{corollary}
		\label{the_depth}
		Let $m\ge 2$ and $J=(f_1,f_2,\ldots,f_m) \subset T$ be a homogeneous ideal, where $\gcd(f_1,\ldots,f_m)=1$.
		Then $\grade(J) \ge 2 $.
		\begin{proof}
			We choose two relatively prime elements $p$ and $q$ from the previous \autoref{lemma_gcd}. Then $p$ is regular on $T$, and $q$ is regular on $T/p$ because $\gcd(p,q)=1$.
			Therefore $\{p,q\}$ is a regular sequence and $\grade(J) \ge 2$.
		\end{proof}
	\end{corollary}

	As consequence of our translation of the condition $\gcd(f_1,\ldots,f_m)=1$ in terms of $\grade(J)\ge 2$, we obtain the following upper bound for the regularity of $J$.
	
	\begin{proposition}
		\label{bound_regularity}
		Let $m\ge2$ and $J=(f_1,f_2,\ldots,f_m) \subset T$ be a homogeneous ideal, where $\gcd(f_1,\ldots,f_m)=1$ and $\deg(f_1),\ldots,\deg(f_m) \le d$.
		Then the regularity is bounded by $\reg(J) \le 3d-2.$
		\begin{proof}
			If we prove $\dim(T/J) \le 1$, then from \cite[Theorem 1.9.4]{REGULARITY} we get  $\reg(J)=\reg(T/J)+1\le 3d - 2$.		
			Since $T$ is a Cohen-Macaulay ring we obtain  $\HT(J)=\grade(J) \ge 2$.		
			Finally, 
			$
			\text{ht}(J) + \dim(T/J) = \dim(T) = 3
			$
			implies that $\dim(T/J) \le 1$.
		\end{proof}
	\end{proposition}

	The following proposition uses the Koszul complex in order to relate the Betti numbers of $J$ with the Hilbert function of $J$. 
	
	\begin{proposition}
		\label{bound_graded_betti}
		Let $J=(f_1,f_2,\ldots,f_m) \subset T$ be a homogeneous ideal. Then
		$$\beta_{2,p}(J)=\dim_{\kk}(\Tor_2^T(J,\kk)_p) \le H_{J}(p-2)-H_{J}(p-3).$$
		\begin{proof}
			Let $\xx=\{s,t,u\}$, we consider the Koszul complex $K(\xx; J) = K(\xx) \otimes_T J$:
			\begin{equation*}
				0 \rightarrow
				J \otimes_T\bigwedge\nolimits^{\!3}T(-3)^3 \xrightarrow{id \otimes_T d_3}
				J \otimes_T\bigwedge\nolimits^{\!2}T(-2)^3 \xrightarrow{id \otimes_T d_2} J \otimes_T\bigwedge\nolimits^{\!1}T(-1)^3 \xrightarrow{id \otimes_T d_1} 
				J \otimes_T\bigwedge\nolimits^{\!0}T^3  \rightarrow 
				0.
			\end{equation*}
			We need to compute in the graded part $(J \otimes_T\bigwedge\nolimits^{\!2}T(-2)^3)_p=J_{p-2} \otimes_{\kk}\bigwedge\nolimits^{\!3}\kk^3$,
			so we only take the complex  
			\begin{align*}
				0 \rightarrow
				J_{p-3} \otimes_{\kk}\bigwedge\nolimits^{\!3}\kk^3 \xrightarrow{(id\otimes_T d_3)_{p-3}}
				J_{p-2} \otimes_{\kk}&\bigwedge\nolimits^{\!2}\kk^3 \xrightarrow{(id\otimes_T d_2)_{p-2}}\\
				&J_{p-1} \otimes_{\kk}\bigwedge\nolimits^{\!1}\kk^3 \xrightarrow{(id\otimes_T d_1)_{p-1}} 
				J_{p}\otimes_{\kk}\bigwedge\nolimits^{\!0}\kk^3  \rightarrow 
				0,
			\end{align*}		
			and we get the formula 
			$$
			\Tor_2^T(J,\kk)_p\cong H_2K(\xx;J)_p = \frac{\Ker(id\otimes_Td_2)_p}{\im(id\otimes_Td_3)_p}
			=\frac{\Ker((id\otimes_Td_2)_{p-2})}{\im((id\otimes_Td_3)_{p-3})}.
			$$
			Then using the fact that $\Ker((id\otimes_Td_2)_{p-2})$ and $\im((id\otimes_Td_3)_{p-3})$ are $\kk$-vector spaces, we can compute  
			$
			\beta_{2,p}(J)=\dim_{\kk}(\Ker((id\otimes_Td_2)_{p-2}))-\dim_{\kk}(\im((id\otimes_Td_3)_{p-3})).
			$ 
			
			From \autoref{pd_ineq_ideal} we know that $\pd_T(J)\le 2$, then we have that $H_3K(\xx;J) \cong \Tor_3^T(J,\kk)=0$ and so $\Ker(id\otimes_Td_3)=0$. 
			From this we conclude that $(id\otimes_Td_3)_{p-3}$ is an injective map and $\dim_{\kk}(\im((id\otimes_Td_3)_{p-3}))=\dim_{\kk}(J_{p-3} \otimes_{\kk}\bigwedge\nolimits^{\!3}\kk^3)=H_{J}(p-3)$.
			
			Let $h_{12} \otimes_{\kk} e_1\wedge e_2+h_{13} \otimes_{\kk} e_1\wedge e_3+h_{23} \otimes_{\kk} e_2\wedge e_3 \in \Ker((id\otimes_T d_2)_{p-2})$. By applying the differential map of the Koszul complex we have 
			$$
			sh_{12}\otimes_{\kk}e_2-th_{12}\otimes_{\kk}e_1+
			sh_{13}\otimes_{\kk}e_3-uh_{13}\otimes_{\kk}e_1+
			th_{23}\otimes_{\kk}e_3-uh_{23}\otimes_{\kk}e_2=0.
			$$ 
			From here we deduce the equations
			$
			th_{12}=-uh_{13}, 
			sh_{12}=uh_{23},
			sh_{13}=-th_{23}.
			$
			Therefore one of the terms can completely determine the other two. This simple fact implies the inequality $\dim_{\kk}(\Ker((id\otimes_Td_2)_{p-2})) \le H_{J}(p-2),$
			and concludes the proof of the proposition.
		\end{proof}
	\end{proposition}
	
	\begin{corollary}
		\label{second_betti}
		Let $m\ge2$ and $J=(f_1,f_2,\ldots,f_m) \subset T$ be a homogeneous ideal, with $\gcd(f_1,\ldots,f_m)=1$ and $\deg(f_1),\ldots,\deg(f_m) \le d$. Then
		$\beta_2(J) \le H_J(\reg(J)) \le H_{J}(3d-2).$
		\begin{proof}
			We have that $\beta_{2,p}=0$ for $p >  \reg(J)+2$. Then we compute 
			$$
			\beta_2(J) = \sum_{p=1}^{\reg(J)+2} \beta_{2,p}(J) \le \sum_{p=1}^{\reg(J)+2} (H_{J}(p-2)-H_{J}(p-3))= H_J(\reg(J))\le H_{J}(3d-2).
			$$ 
			The last inequality is obtained from \autoref{bound_regularity}.
		\end{proof}
	\end{corollary}
	
	\begin{proof}[Proof of \autoref{bounds}.]
		\textit{(i)} The upper bound for the regularity has already been proved in \autoref{bound_regularity}.
		
		\textit{(ii)} Follows from the additivity of the rank function.
		
		\textit{(iii)}
		We know that the number of monomials of degree $d$ in $\kk[s,t,u]$ is $\binom{d+2}{2}$, hence from \autoref{second_betti} we get the upper bound $\beta_2(J)  \le H_J(\reg(J)) \le H_J(3d-2) \le \binom{3d}{2}.$ 
		
		Now we add the extra condition that $J=(f_1,f_2,\ldots,f_m)$ is generated by $m$ polynomials of the same degree $d$.
		Hence for any $p \ge d$ we have that the $\kk$-vector space $J_p$ is generated by elements of the form $gf_i\; (1\le i\le m)$ where $g$ is a monomial of degree $p-d$. 
		So we have that the Hilbert function of $J$ is bounded by 
		$
		H_{J}(p) \le m\binom{p-d+2}{2}.
		$
	\end{proof}
	
	Now for the second part of this section we work with an ideal $J = (g_1, g_2, g_3, g_4) \in S$, such that $d=\deg(g_1)=\cdots=\deg(g_4)$ and $\HT(J)=3$.
	
	\begin{theorem}
		\label{HH3}
		Let $J=(g_1,g_2,g_3,g_4) \subset S$ be a homogeneous ideal with $d=\deg(g_1)=\cdots=\deg(g_4)$ and $\HT(J)=3$. Then $
		\beta_1(J) \le 2d+2$ and $
		\beta_2(J) \le 2d-1$.
	\end{theorem}
	
	The proof of \autoref{HH3} is divided in some steps that are given below.
	
	\begin{remark}
		From the Unmixedness Theorem and the fact that $\RR$ is an infinite field, we can find a complete intersection inside $J$ (see \cite[Lemma A.10]{SZANTO}, \cite[Theorem 125]{KAP}).
		Explicitly, there exist scalars $\alpha_{ij} \in \RR$ that give us the following sort of triangular transformation 
		\begin{align*}
			h_1&=   g_1 + \alpha_{12}g_2 + \alpha_{13}g_3 + \alpha_{14}g_4, \\
			h_2&=   g_2  + \alpha_{23}g_3 +\alpha_{24}g_4,\\
			h_3&=  g_3 + \alpha_{34}g_4,\\
			h_4&=  g_4, 
		\end{align*} 
		where $\{h_1,h_2,h_3\}$ is a complete intersection.
		Therefore, we can assume that $J=(h_1,h_2,h_3,h_4)$, where $\{h_1,h_2,h_3\}$ is a complete intersection and $d=\deg(h_1)=\cdots=\deg(h_4)$.
		Also, we can suppose that $h_4 \not\in (h_1,h_2,h_3)$, because in case $J=(h_1,h_2,h_3)$ then the minimal free resolution of $S/J$ can be obtained with the Koszul complex, that trivially  satisfies the result of \autoref{HH3}.
	\end{remark}
	
	We shall take a similar approach to \cite{ROIG} using a process of linkage or liaison. We make the observation that $J=(h_1,h_2,h_3,h_4)$ can be linked to a Gorenstein ideal $G$  (see \cite[Corollary 5.19]{LIAISON}, \cite[Proposition 5.2]{GORH3}) via the complete intersection $K=(h_1,h_2,h_3)$, i.e., $G=(K:J)$.
	
	The minimal free resolution of $S/K$ is given by the Koszul complex.
	Using Buchsbaum and Eisenbud's structure theorem for height 3 Gorenstein ideals \cite[Theorem 2.1]{GORH3}, the minimal free resolution of $S/G$ has the form
	$$
	0 \rightarrow S(-s-3) \xrightarrow{g^{*}} \bigoplus_{i=1}^{m}S(-p_i) \xrightarrow{f} \bigoplus_{i=1}^{m}S(-q_i) \xrightarrow{g} S \rightarrow S/G \rightarrow 0,
	$$
	where $s$ is the socle degree of $G$ (the largest $k$ such that ${(S/G)}_k\neq 0$), $m$ is odd, $f$ is alternating, and $G=\text{Pf}_{m-1}(f)$ (the ideal generated by the $(m-1)$-th Pfaffians of $f$).
	
	\begin{lemma}
		\label{Socle_Deg}
		The socle degree of $S/G$  is $s=2d-3$.
		\begin{proof}
			Since $K$ is a complete intersection we know that the socle degree of $S/K$ is $3d-3$. 
			We have that the Hilbert function of an Artinian Gorenstein algebra is symmetric, also we can relate the Hilbert functions of $S/K$, $S/G$ and $S/J$ (see \cite[Theorem 2.10, page 308]{SIXCOMM}, \cite{ROIG}) in the following way
			$$
			H_{S/G}(t) = H_{S/K}(3d-3-t) - H_{S/J}(3d-3-t).
			$$
			Then for any $t > 2d-3$ we have $3d-3-t < d$ and $H_{S/J}(3d-3-t)=H_{S/K}(3d-3-t)=\binom{3d-1-t}{2}$, also we can easily check that $H_{S/G}(2d-3)=1$. 
		\end{proof}
	\end{lemma}
	
	\begin{proof}[Proof of \autoref{HH3}.]
		By using \autoref{Socle_Deg} we substitute the socle degree of $S/G$  in its minimal free resolution, and from the canonical map $S/K \rightarrow S/G$ we can lift a comparison map
		\begin{center}					
			\begin{tikzpicture}		
			\matrix (m) [matrix of math nodes,row sep=3em,column sep=2.5em,minimum width=2em,text height=1.5ex, text depth=0.25ex]
			{
				0&S(-3d) & S(-2d)^3 & S(-d)^3 & S & S/K & 0 \\
				0 & S(-2d)  &   \bigoplus_{i=1}^{m}S(-p_i)    & \bigoplus_{i=1}^{m}S(-q_i) & S & S/G & 0. \\
			};
			\path[-stealth]	
			(m-1-6) edge node [right] {$can$} (m-2-6)
			(m-1-5) edge node [right] {$\phi_0=id$} (m-2-5)
			(m-1-4) edge node [right] {$\phi_1$} (m-2-4)
			(m-1-3) edge node [right] {$\phi_2$} (m-2-3) 
			(m-1-2) edge node [right] {$\phi_3$} (m-2-2)	
			(m-1-1) edge (m-1-2)	
			(m-2-1) edge (m-2-2)	
			(m-1-6) edge (m-1-7)	
			(m-2-6) edge (m-2-7)
			(m-1-2) edge node [above] {$d_3$} (m-1-3)
			(m-1-3) edge node [above] {$d_2$} (m-1-4)
			(m-1-4) edge node [above] {$d_1$} (m-1-5)	
			(m-1-5) edge  (m-1-6)	
			(m-2-2) edge node [above] {$g^{*}$} (m-2-3)
			(m-2-3) edge node [above] {$f$} (m-2-4)
			(m-2-4) edge node [above] {$g$} (m-2-5)	
			(m-2-5) edge  (m-2-6)		
			;			
			\end{tikzpicture}
		\end{center}
		
		With a dual mapping cone construction (\cite{GORH3}, \cite{ROIG}) we can obtain the following free resolution for $S/J$ (not necessarily minimal)
		$$
		0 \rightarrow \bigoplus_{i=1}^mS(-3d+q_i) \rightarrow 
		\begin{array}{c}
		S(-2d)^3\\
		\bigoplus \\
		\bigoplus_{i=1}^mS(-3d+p_i)
		\end{array} \rightarrow S(-d)^4 \rightarrow S \rightarrow
		S/J \rightarrow 0.
		$$
		Thus we have $\beta_2(J) \le m$ and $\beta_1(J) \le m+3$.
		In \cite[Theorem 3.3]{DIESEL} it is proved that given the smallest degree $k$ of the generators of $G$ (i.e., $k$ is the first position in which $H_{S/G}(k) < \binom{k+2}{2}$) then $m\le 2k+1$.
		
		Since $S/G$ has socle degree $2d-3$ and  its Hilbert function is symmetric, we have $H_{S/G}(d-2)=H_{S/G}(d-1)$. Hence $k \le d-1$ because otherwise we get the contradiction $\binom{d}{2} = \binom{d+1}{2}$. Therefore, we have obtained $\beta_2(J) \le 2d-1$ and $\beta_1(J) \le 2d+2$.
	\end{proof}
	
	\begin{remark}
		\label{general_form}
		An interesting fact proved in \cite[Corollary 4.4]{ROIG}, is that when the ideal $J=(g_1,g_2,g_3,g_4)$ is a general Artinian almost complete intersection of type $(d,d,d,d)$, then the minimal free resolution can be given explicitly. This means that $J$ is generated by ``generically chosen'' polynomials $g_1,g_2,g_3,g_4$, where $(g_1,g_2,g_3)$ are a complete intersection, $g_4 \not\in (g_1,g_2,g_3)$, and $d=\deg(g_1)=\cdots=\deg(g_4)$. The  minimal free resolution of $S/J$ in this case is
		$$
		0 \rightarrow S(-2d-1)^d \rightarrow 
		\begin{array}{c}
		S(-2d)^3\\
		\bigoplus \\
		S(-2d+1)^d
		\end{array} \rightarrow S(-d)^4 \rightarrow S \rightarrow
		S/J \rightarrow 0.
		$$
		The term ``generically chosen'' means that $g_1,g_2,g_3,g_4$ belong to a suitable dense open subset of $S_d \times S_d \times S_d \times S_d$ in the Zariski topology.
	\end{remark}
	
	\section{Projective dimension two}
	\label{section4}
	
	From \autoref{pd_ineq_ideal} we know that $\pd(\II) \le 2$.
	Here we deal with the remaining case $\pd(\II)=2$, because $\pd(\II)=1$ was studied in \cite{STRONG_MU_BASIS}.
	In the rest of this paper we use the following notation.
	\begin{notation}
		Let $\RR$ be an infinite field, $R$ be the polynomial ring $R=\RR[s,t]$ and $S$ be the polynomial ring $S=\RR[s,t,u]$.
	\end{notation}

	From \autoref{Euler_char} we get a free resolution 
	\begin{equation}
		\label{resolut_pd_2}
		0 \rightarrow S^{a} \xrightarrow{\hat{d_2}} S^{a+3} \xrightarrow{\hat{d_1}}  S^4 \xrightarrow{[b_1,b_2,b_3,b_4]} \II \rightarrow 0,  
	\end{equation}
	and now we want to compute the value of $a$. Here we are using an abuse of notation, because  we should write
	\begin{equation}
		\label{resolut_pd_2_}
		0 \rightarrow \bigoplus_{i=1}^{a} S(-p_i) \xrightarrow{\hat{d_2}} \bigoplus_{i=1}^{a+3} S(-q_i) \xrightarrow{\hat{d_1}}  S(-d)^4 \xrightarrow{[b_1,b_2,b_3,b_4]} \II \rightarrow 0,
	\end{equation}
	if we want to take care of the grading.
	
	In \autoref{resolut_pd_2} we do not know if ${b_1,b_2,b_3,b_4}$ is a minimal system of generators. 
	But in the next step of finding the resolution \autoref{resolut_pd_2} of $\II$, we can choose a minimal system of generators for $\Syz(b_1,b_2,b_3,b_4)$ because it is a graded module.
	Therefore, from \cite[Theorem 7.3]{GRADED} we can assure that $\im(\hat{d_2}) \subseteq \mm S^{a+3}$, where $\mm=(s,t,u)$ is the irrelevant ideal.
	By exploiting the condition $\im(\hat{d_2}) \subset \mm S^{a+3}$ we will \textit{``adapt''} the upper bounds obtained in the previous section to  \autoref{resolut_pd_2}.
	
	\begin{lemma}
		\label{bound_grading}
		For the resolution \autoref{resolut_pd_2} (more specifically \autoref{resolut_pd_2_}) we have that
		\begin{enumerate}[(i)]
			\item $\max_{1\le i \le a+3}(q_i) \le 3d-1$,
			\item $\max_{1\le i \le a}(p_i) \le 3d$,
			\item $a = \beta_2(\II).$
		\end{enumerate}
		\begin{proof}		
			Here we use the key fact that for a graded free $S$-module $F = \bigoplus_{i = 1}^r S(-\alpha_i)$ we have $(F \otimes_S \RR)_p=0$ if and only if $p \neq \alpha_i$ for all $1 \le i \le r$.
			
			\textit{(i)}
			Let $p> 3d-2+1=3d-1$.  
			The upper bound $\reg(\II) \le 3d-2$ (\autoref{bounds}$(i)$) yields that $\beta_{1,p}=0$, and this implies that 
			$\Tor_1^S(\II,\RR)_p=0$.
			The condition $\im(\hat{d_2}) \subseteq \mm S^{a+3}$ gives us that $\im(\hat{d_2} \otimes_S \RR) =0$ and so we get $\Ker((\hat{d_1} \otimes_S \RR)_p)=\Tor_1^S(\II,\RR)_p=0$.
			Since $(\hat{d_1} \otimes_S \RR)_p$
			is an injective map and $(S(-d)^4 \otimes_S \RR)_q=0$ for $q > d$, then we conclude
			$$
			{\left(\left(\bigoplus_{i=1}^{a+3}S(-q_i)\right)\bigotimes_{S} \RR\right)}_p=0.
			$$ 
			Therefore we have the inequality $\max_{1\le i \le a+3}(q_i) \le 3d-1$.
			
			\textit{(ii), (iii)}	Deleting $\II$ from \autoref{resolut_pd_2} and applying the tensor product $\otimes_S \RR$ we get the complex 
			$$
			0 \rightarrow \RR^{a} \xrightarrow{0} \RR^{a+3} \xrightarrow{\hat{d_1} \otimes_S F}  \RR^4  \rightarrow 0.
			$$	
			So $a=\dim_{\RR}(\Tor_2^S(\II,\RR))=\beta_2(\II)$, and the grading of the module $S^a$ is just like the one for a minimal free resolution, i.e. $\le (3d-2)+2=3d$.
		\end{proof}
	\end{lemma}
	
	So in this case the resolution of $\II$ is of the form 
	$$
	0 \rightarrow S^{\beta_2(\II)} \xrightarrow{\hat{d_2}} S^{\beta_2(\II)+3} \xrightarrow{\hat{d_1}}  S^4 \xrightarrow{[b_1,b_2,b_3,b_4]} \II \rightarrow 0,
	$$
	where the polynomials in the entries of the matrices $\hat{d_1}$ and $\hat{d_2}$ have degree bounded by $3d-1-d=2d-1$, and  $3d-d=2d$ respectively.
	
	\begin{remark}
		\label{Upper_bound_deg_Syz}
		We apply the tensor product with $\otimes_S S/(u-1)$ to obtain the exact sequence (see \cite[Corollary 19.8]{EISENBUD_COMM} or \cite[Proposition 1.1.5]{BRUNS_HERZ})
		$$
		0 \rightarrow R^{\beta_2(\II)} \xrightarrow{d_2} R^{\beta_2(\II)+3} \xrightarrow{d_1}  R^4 \xrightarrow{[a_1,a_2,a_3,a_4]} I \rightarrow 0,
		$$
		where $d_1=\hat{d_1} \otimes_S S/(u-1)$ and $d_2=\hat{d_2} \otimes_S S/(u-1)$ are matrices with entries in $R$ bounded in degree by $2d-1$ and $2d$ respectively.
	
		For the rest of this section we shall work with the exact sequence 
		\begin{equation}
			\label{split_exact_seq}
			0 \rightarrow R^{\beta_2(\II)} \xrightarrow{d_2} R^{\beta_2(\II)+3} \xrightarrow{d_1} \Syz(a_1,a_2,a_3,a_4) \;(\subset R^4) \rightarrow 0,
		\end{equation}
		which is a split exact sequence because we know that $\Syz(a_1,a_2,a_3,a_4)$ is a free module.
	\end{remark}
	
	An $m \times n\;(m>n)$ polynomial matrix $A \in R^{m \times n}$ is said to be \textit{unimodular} if it satisfies one of the following equivalent conditions (see \cite{SERRE_CONJ})
	\begin{enumerate}[(i)]
		\item $A$ can be completed into an invertible $m \times m$ square matrix.
		\item there exists an $n \times m$ polynomial matrix $B \in R^{n \times m}$ such that $AB=I_m$.
		\item there exists an $n \times m$ polynomial matrix $B \in R^{n \times m}$ such that $BA=I_n$.
		\item the ideal generated by the $n \times n$ minors of A is equal to $R$.
	\end{enumerate}
	
	We define the degree of a matrix $M=(a_{ij}) \in R^{m \times n}$ as the maximum degree of the polynomial entries of $M$, i.e., $\deg(M)=\max(\deg(a_{ij}))$. For an \textit{``effective''} solution of completing a unimodular matrix we are going to use the following result from \cite{SERRE_ART}.
	
	\begin{theorem}
		\label{CANIGLIA_RES}
		Let $F \in R^{m \times n} \; (m < n)$ be a unimodular matrix. 
		Then there exists a square matrix $M \in R^{n \times n}$ such that
		\begin{enumerate}[(i)]
			\item $M$ is unimodular,
			\item $FM=[I_m,0] \in R^{m \times n}$,
			\item $\deg(M) \le 2D(1+2D)(1+D^{4})(1+D)^{4}$, where $D = m(1 + \deg(F))$.  
		\end{enumerate}
		\begin{proof}
			See the \hyperref[appendix]{Appendix} for a discussion.
		\end{proof} 
	\end{theorem}
	
	This previous result is given for completing rows (i.e., $m < n$), but we want to complete columns (i.e., $m > n$). By simply taking transpose in $(ii)$ of the previous theorem we get the following corollary.
	
	\begin{corollary}
		\label{completing_cols}
		Let $F \in R^{m \times n} \; (m > n)$ be a unimodular matrix. 
		Then there exists a square matrix $M \in R^{m \times m}$ such that
		\begin{enumerate}[(i)]
			\item $M$ is unimodular,
			\item $MF=\left[ \begin{array}{c}
			I_n \\
			0
			\end{array}\right] \in R^{m \times n}$,
			\item $\deg(M) \le 2D(1+2D)(1+D^{4})(1+D)^{4}$, where $D = n(1 + \deg(F))$.  
		\end{enumerate}
	\end{corollary}
	
	\noindent
	For notational purposes we make the following conventions 
	\begin{compactitem}
		\item we use just $\beta_2$ instead of $\beta_2(\II)$,
		\item $m = \beta_2+3$ and $n=\beta_2$,
		\item $F \in R^{m \times n}$ denotes the $m \times n$ matrix corresponding with the map $d_2$,
		\item $\gamma_2 = \deg(F)$,
		\item $G \in R^{4 \times m}$ denotes the $4 \times m$ matrix corresponding with the map $d_1$,
		\item $\gamma_1 =\deg(G)$,
		\item $D = n(1+\deg(F)) = \beta_2(1+\gamma_2)$,
	\end{compactitem}
	thus we end up with the following short exact sequence
	\begin{equation}
		\label{split_seq_final}
		0 \rightarrow R^n \xrightarrow{F} R^m \xrightarrow{G} \Syz(a_1,a_2,a_3,a_4)\;(\subset R^4) \rightarrow 0.
	\end{equation}
	Since this sequence splits, there exists a matrix $H \in R^{n \times m}$ with $HF=I_n$ and so the matrix $F$ is unimodular. 
	
	\begin{proposition}
		\label{find_basis}
		From the exact sequence \autoref{split_seq_final}
		we can get a basis for $\Syz(a_1,a_2,a_3,a_4)$ made of three vectors $\mathbf{p},\mathbf{q},\mathbf{r} \in R^4$, with 
		\begin{align*}
			\max(\deg(\mathbf{p}),\deg(\mathbf{q}),\deg(\mathbf{r})) &\le \gamma_1(\beta_2+2)2D(1+2D)(1+D^{4})(1+D)^{4}\\
			\le 2\gamma_1(\beta_2+2)\beta_2&(1+\gamma_2)(1+2\beta_2(1+\gamma_2))(1+{(\beta_2(1+\gamma_2))}^{4})(1+\beta_2(1+\gamma_2))^{4}.			
		\end{align*}
		\begin{proof}
			We can get a matrix $M \in R^{m \times m}$  that satisfies \textit{(i), (ii), (iii)} from \autoref{completing_cols}.		
			Let $N \in R^{m \times m}$ be the inverse matrix of $M$, then we have that
			$\deg(N) \le (m-1)\deg(M)$,
			because the determinant of every $(m-1)\times(m-1)$-minor is a polynomial of degree at most $m-1$ in terms of the entries of $M$.	
			Also, from the item $(ii)$ of \autoref{completing_cols} we have that
			$$
			F = N \left[ \begin{array}{c}
			I_n \\
			0
			\end{array}\right]. 
			$$
			
			We know that $N$ is an automorphism for $R^m=N(\text{span}(e_1,\ldots,e_n,e_{n+1},e_{n+2},e_{n+3}))$ and that 
			$
			\Ker(G) = \im(F) = N(\text{span}(e_1, \ldots, e_n)), 
			$		
			where $e_i$ is the $i$-th column vector of $R^m$. Hence we have
			$$
			\Syz(a_1,a_2,a_3,a_4)=\im(G)=G(R^m)=GN(\text{span}(e_1,\ldots,e_m))=GN(\text{span}(e_{n+1},e_{n+2},e_{n+3})),
			$$
			and we define the basis for $\Syz(a_1,a_2,a_3,a_4)$ as
			\begin{align*}
				\mathbf{p} &= GNe_{n+1},\\
				\mathbf{q} &= GNe_{n+2},\\
				\mathbf{r} &= GNe_{n+3}.
			\end{align*}
			Finally, we obtain the result
			\begin{align*}
				\max(\deg(\mathbf{p}),\deg(\mathbf{q}),\deg(\mathbf{r})) &\le \deg(G)\deg(N) \le \gamma_1(\beta_2+2)2D(1+2D)(1+D^{4})(1+D)^{4} \\	
				\le 2\gamma_1(\beta_2+2)\beta_2&(1+\gamma_2)(1+2\beta_2(1+\gamma_2))(1+{(\beta_2(1+\gamma_2))}^{4})(1+\beta_2(1+\gamma_2))^{4}. \qedhere
			\end{align*}
		\end{proof}
	\end{proposition}
	
	The following theorem contains the main result of this paper, and gives different degree bounds for the generators of the basis depending on the type of exact sequence (presentation) \autoref{split_seq_final} that we can obtain.
	
	\begin{theorem}
		\label{result_pd_2}
		Given the data $\{a_1,a_2,a_3,a_4\}$ defining \autoref{rational_surface} where 
		$$
		d = \max_{1\le i \le4}(\deg(a_i)) \quad \text{ and } \quad \gcd(a_1,a_2,a_3,a_4)=1.
		$$  
		Then, the following statements hold: 
		\begin{enumerate}[(i)]
			\item There exists a basis for $\Syz(a_1,a_2,a_3,a_4)$ with polynomials bounded in the order of $O(d^{33})$.
			\item If the homogenized ideal $\II$ (obtained in \autoref{homogenization}) has height $\HT(\II)=3$, then there exists a basis with degree bounded by $O(d^{22})$.
			\item If the homogenized ideal $\II$ is a ``general'' almost complete intersection (i.e, like in \autoref{general_form}), then there exists a basis with degree bounded by $O(d^{12})$.
			\item If the homogenized ideal $\II$ has projective dimension $\pd(\II)=1$, then there exists a basis with degree bounded by $d$.
		\end{enumerate}
		\begin{proof}
			$(i)$ By \autoref{bound_grading} we know that $\gamma_1=\deg(G) \le 2d-1$ and $\gamma_2=\deg(F)\le 2d$, from  \autoref{bounds} we have $\beta_2 \le 4\binom{2d}{2} \in O(d^2)$.
			Therefore substituting in the formula  obtained in \autoref{find_basis} we get a basis bounded by $O(d^{33})$.
			
			$(ii)$ In this case from \autoref{HH3} we have $\beta_2 \le 2d-1$. So we can reduce the upper bound to $O(d^{22})$.
			
			$(iii)$ In \autoref{general_form} we saw that when $\II$ is a general Artinian almost complete intersection then $S/\II$ has a very special minimal free resolution, from which we can obtain $\gamma_1=\deg(G)=d$, $\gamma_2=\deg(F)=2$ and $\beta_2=d$.
			Therefore we obtain an upper bound in the order of $O(d^{12})$.
			
			$(iv)$ From \cite[\S 5.1]{STRONG_MU_BASIS}, the resolution of $\II$ is given by 
			$$
			0 \rightarrow S(-d-\mu_1)\oplus S(-d-\mu_2)\oplus S(-d-\mu_3) \rightarrow S(-d)^4 \rightarrow \II \rightarrow 0, 
			$$
			where $\mu_1+\mu_2+\mu_3=d$.
			Then the same dehomogenization of \autoref{Upper_bound_deg_Syz} gives us the result.
		\end{proof}
	\end{theorem}

	\section{Discussions of the upper bounds and a related problem}
	\label{section5}
	
	In this short section we discuss the sharpness of the upper bounds obtained in \autoref{result_pd_2}.
		
	First, from \autoref{Upper_bound_deg_Syz} we obtain the interesting fact that the syzygy module $\Syz(a_1,a_2,a_3,a_4)$ can be generated by elements of degree bounded by $2d-1$.
	In our particular case,
	this is almost identical to a remarkable result of Lazard (see \cite{LAZARD_UPPER_BOUND_NOTE,LAZARD_ALG_LIN}).
	Let $S(n,d)$ be the least integer such that the module of syzygies $\Syz(h_1,h_2,\ldots,h_k)$ can be generated by elements of degree at most $S(n,d)$, where $h_1,h_2,\ldots,h_k$ are arbitrary polynomials in $n$ variables and degree at most $d$ (this definition is independent of the particular polynomials $h_i$'s).
	In \cite{LAZARD_UPPER_BOUND_NOTE} there is an important general upper bound for $S(n,d)$, and also the following 
	$$
	S(2,d)\le 2d - \min(2, d)
	$$
	sharp upper bound
	(see \cite[Proposition 5, Proposition 10]{LAZARD_ALG_LIN}).
	So, the upper bound that we obtained for the degree of the generators of $\Syz(a_1,a_2,a_3,a_4)$ is ``almost sharp''.
		
	The main obstacle is that in the general case of non-graded modules we do not have a well-defined concept of ``minimal set of generators'' (see e.g. \cite[pages 236 and 237, Excercise 4]{CLO_USING_AG}).
	Due to this fact, it is unclear to the author how to obtain a basis of $\Syz(a_1,a_2,a_3,a_4)$ if we are given a set of generators. 
	After obtaining the results of \autoref{bound_grading} and \autoref{Upper_bound_deg_Syz}, we ``only needed'' to solve a particular case of the following problem.
	
	\begin{problem}
		\label{PROBLEM}
		Find a function $f:\NN^2\rightarrow \NN$, such that for any free module $F \subset {\kk[x_1,\ldots,x_n]}^m$ of rank $r < m$ and generated by a set of elements $\{v_1,v_2,\ldots,v_k\} \subset {\kk[x_1,\ldots,x_n]}^m$ with $\deg(v_i) \le d$, then there exists a basis $\{w_1,w_2,\ldots,w_r\} \subset {\kk[x_1,\ldots,x_n]}^m$ of $F$ satisfying the condition 
		$$
		\deg(w_i) \le f(d,k).
		$$		
	\end{problem}

	In this paper we use an effective version of the Quillen-Suslin Theorem to solve the problem above.
	We remark that the best known upper bounds for the effective Quillen-Suslin Theorem are given in \cite{SERRE_ART} (see \cite[page 715, Remark (1)]{LOMBARDI}).
	
	In conclusions, the sharpness of the upper bounds in \autoref{result_pd_2} depends mostly in our ability to solve \autoref{PROBLEM}.
	As a general opinion, we think that they can be improved.
	
	\section{Example}
	\label{section6}
	
	The aim of this example is to show some computational aspects of the case studied in this paper (i.e., $\pd(\II)=2$).
    With a simple example, we make all the steps of our method for computing a $\mu$-basis.
	
	\begin{example}
		Find a $\mu$-basis for the rational surface parametrization
		$$
		P(s,t)=(s^2,t^2,s^2-1,s^2+1).
		$$
		\begin{proof}
			Using a computer algebra system like \textit{Singular} (\cite{SINGULAR}), we get the following free resolution 
			$$
			0 \rightarrow S \xrightarrow{\left(\begin{array}{c}
				0 \\
				s^2  \\
				-u^2  \\
				-t^2  
				\end{array}\right)}
			S^4 \xrightarrow{\left(\begin{array}{cccc}
				-2 & -t^2 & -t^2 & u^2-s^2\\
				0  & u^2  &  s^2 & 0\\
				1  & t^2  &  0   & s^2\\
				1  & 0    &  0   & 0
				\end{array}\right)} S^4 \xrightarrow{\left(\begin{array}{cccc}
				s^2 & t^2 & s^2-u^2 & s^2+u^2   
				\end{array}\right)}
			\II \rightarrow 0.
			$$
			Substituting $u=1$ and cutting the resolution, we obtain 
			$$
			0 \rightarrow R \xrightarrow{\left(\begin{array}{c}
				0 \\
				s^2  \\
				-1  \\
				-t^2  
				\end{array}\right)}
			R^4 \xrightarrow{\left(\begin{array}{cccc}
				-2 & -t^2 & -t^2 & 1-s^2\\
				0  & 1  &  s^2 & 0\\
				1  & t^2  &  0   & s^2\\
				1  & 0    &  0   & 0
				\end{array}\right)} \Syz(s^2, t^2, s^2-1, s^2+1) \;(\subset R^4)  \rightarrow 0.
			$$
			
			Proceeding as in \autoref{find_basis}, we have to complete the unimodular column $(0, s^2, -1, -t^2)^{t}$ into an invertible matrix $N \in R^{4 \times 4}$. 
			For this we can check that the following matrix 
			$$
			N = \left(\begin{array}{cccc}
			0   & 0  &  1 & 0\\
			s^2 & 1  &  0 & 0\\
			-1  & 0  &  0 & 0\\
			t^2 & 0  &  0 & 1
			\end{array}\right)
			$$
			has determinant $1$.
			Therefore, a $\mu$-basis for $P(s,t)=(s^2,t^2,s^2-1,s^2+1)$ is given by the vectors
			$$	
			\mathbf{p} = \left(\begin{array}{cccc}
			-2 & -t^2 & -t^2 & 1-s^2\\
			0  & 1  &  s^2 & 0\\
			1  & t^2  &  0   & s^2\\
			1  & 0    &  0   & 0
			\end{array}\right)
			\left(\begin{array}{cccc}
			0   & 0  &  1 & 0\\
			s^2 & 1  &  0 & 0\\
			-1  & 0  &  0 & 0\\
			t^2 & 0  &  0 & 1
			\end{array}\right)
			\left(\begin{array}{c}
			0\\
			1\\
			0\\
			0
			\end{array}\right)
			=
			\left(\begin{array}{c}
			-t^2\\
			1\\
			t^2\\
			0
			\end{array}\right),
			$$
			$$
			\mathbf{q} = \left(\begin{array}{cccc}
			-2 & -t^2 & -t^2 & 1-s^2\\
			0  & 1  &  s^2 & 0\\
			1  & t^2  &  0   & s^2\\
			1  & 0    &  0   & 0
			\end{array}\right)
			\left(\begin{array}{cccc}
			0   & 0  &  1 & 0\\
			s^2 & 1  &  0 & 0\\
			-1  & 0  &  0 & 0\\
			t^2 & 0  &  0 & 1
			\end{array}\right)
			\left(\begin{array}{c}
			0\\
			0\\
			1\\
			0
			\end{array}\right)
			=
			\left(\begin{array}{c}
			-2\\
			0\\
			1\\
			1
			\end{array}\right),
			$$	
			$$	
			\mathbf{r} = \left(\begin{array}{cccc}
			-2 & -t^2 & -t^2 & 1-s^2\\
			0  & 1  &  s^2 & 0\\
			1  & t^2  &  0   & s^2\\
			1  & 0    &  0   & 0
			\end{array}\right)
			\left(\begin{array}{cccc}
			0   & 0  &  1 & 0\\
			s^2 & 1  &  0 & 0\\
			-1  & 0  &  0 & 0\\
			t^2 & 0  &  0 & 1
			\end{array}\right)
			\left(\begin{array}{c}
			0\\
			0\\
			0\\
			1
			\end{array}\right)
			=
			\left(\begin{array}{c}
			1-s^2\\
			0\\
			s^2\\
			0
			\end{array}\right).
			$$
		\end{proof}
	\end{example}

	\section*{Acknowledgments}
	\addcontentsline{toc}{section}{Acknowledgements}
	The author is  grateful to the  support and guidance of Lothar G\"{o}ttsche as his thesis supervisor in the Postgraduate Diploma in Mathematics of ICTP, where a big part of this project was developed.
	The author expresses his gratitude to his current PhD supervisor in the University of Barcelona, Carlos D'Andrea, for suggesting the problem, for a thorough reading of earlier drafts, and for his important help in the culmination of the paper.
	The author thanks Laurent Bus\'{e}, Alicia Dickenstein, Rosa Maria Mir\'{o}-Roig, Pablo Solern\'{o} and Martin Sombra  for helpful discussions and suggestions.
	The author wishes to thank the referees for numerous suggestions to improve the exposition.
	The author thanks David Cox for pointing out his previous work \cite{STRONG_MU_BASIS}.
	The author was funded by the European Union's
	Horizon 2020 research and innovation programme under the Marie Sk\l{}odowska-Curie grant
	agreement No. 675789.

	\begin{appendices}	
		\section*{Appendix}
		\addcontentsline{toc}{section}{Appendix} 
		\label{appendix}
		In this appendix we will discuss the ``effective'' completion of unimodular matrices that we used in \autoref{find_basis}.
		The main result we shall follow from \cite{SERRE_ART} is the following theorem.
		\begin{theorem}
			\label{serre_bound}
			\textit{\cite[Theorem 3.1]{SERRE_ART}} Let $R=\RR[x_1,\ldots,x_n]$ and assume that $F \in R^{r \times s} \; (r < s)$ is unimodular. 
			Then there exists a square matrix $M \in R^{s \times s}$ such that
			\begin{enumerate}[(i)]
				\item $M$ is unimodular,
				\item $FM=[I_r,0] \in R^{r \times s}$,
				\item $\deg(M)=(r(1+\deg(F))^{O(n)}$.  
			\end{enumerate} 
		\end{theorem}
		
		In our case $n=2$ and we will have to make some small \textit{``adjustments''} to find an actual constant. We will follow exactly the same proof as in \cite{SERRE_ART}, and in certain steps we will substitute phrases like $n+3n \in O(n)$ by the exact computation $n+3n=4n$.
		Inside this appendix section by the variable $d$ we denote $d = 1 + \deg(F)$.
		
		\begin{proposition}
			\textit{\cite[Proposition 4.1]{SERRE_ART}}
			Assume that $F \in R^{r \times s} \; (R=\RR[x_1,\ldots,x_n],\; r < s)$ is unimodular. Then there exists a square matrix $M \in R^{s \times s}$ such that:
			\begin{enumerate}[(i)]
				\item $M$ is unimodular,
				\item $FM=[f_{ij}(x_1,\ldots,x_{n-1},0)]$ (i.e., $FM$ is equal to the $r \times s$ matrix obtained by specializing the indeterminate $x_n$ to zero in the matrix $F$),
				\item $\deg(M) \le D(1+2D)(1+D^{2n})(1+D)^{2n}$, with $D = r(1+\deg(F))=rd$.
			\end{enumerate}
			\begin{proof}
				We denote $F(t)$ as the matrix $F(t)=[f_{ij}(x_1,\ldots,x_{n-1},t)]$.
				
				\underline{\textit{Claim 1.}} \textit{\cite[Procedure 4.6, Step 1 and Step 2]{SERRE_ART}} There exists elements $c_1, \ldots, c_N \in \RR[x_1,\ldots,x_{n-1}]$ with $N \le (1+rd)^{2n}$ such that $1 \in (c_1, \ldots, c_N)$. Also we can find elements $a_1,\ldots,a_N \in x_n\RR[x_1,\ldots,x_{n-1}]$, such that 
				$
				x_n = a_1c_1+\ldots+a_Nc_N
				$
				and with 
				$
				\max_{1 \le k \le N} \{\deg(a_kc_k)\} \le 1+(rd)^{2n}.
				$
				
				\underline{\textit{Claim 2.}} \textit{ \cite[Procedure 4.6, Step 3 and Step 4]{SERRE_ART}} For $1 \le k \le N$, let 
				$	
				b_k = \sum_{h=1}^{k} a_hc_h,
				$
				then there exist unimodular matrices $E_k$ with the properties 
				\begin{compactitem}
					\item $F(b_k)E_k = F(b_{k-1})$, 
					\item $\deg(E_k) \le rd(1+2rd)\max\{\deg(b_k),\deg(b_{k-1})\} \le rd(1+2rd)(1+(rd)^{2n})$.
				\end{compactitem}
				
				Therefore we define $M=E_NE_{N-1}\ldots E_1$
				and we have $F(x_n)M=F(0)$, with the upper bound
				\begin{align*}
					\deg(M) &\le Nrd(1+2rd)(1+(rd)^{2n})\le rd(1+2rd)(1+(rd)^{2n})(1+rd)^{2n} \\
					\deg(M) &\le D(1+2D)(1+D^{2n})(1+D)^{2n}.
				\end{align*}
			\end{proof}
		\end{proposition}
		
		\begin{proof}[Proof of \autoref{serre_bound}.]
			For a matrix $F=[f_{ij}(x_1,\ldots,x_n)]$ the substitution of a  variable $x_i$ for $0$ does not increase the degree of the matrix, and keeps the unimodularity.
			Therefore, applying the previous proposition $n$ times and some elementary transformations, we can find an invertible matrix $M \in R^{s \times s}$, with  $FM=[I_r,0]$ and 	$
			\deg(M) \le nD(1+2D)(1+D^{2n})(1+D)^{2n}.
			$
		\end{proof}
		
	\end{appendices}


\begin{thebibliography}{46}
\expandafter\ifx\csname natexlab\endcsname\relax\def\natexlab#1{#1}\fi
\expandafter\ifx\csname url\endcsname\relax
  \def\url#1{\texttt{#1}}\fi
\expandafter\ifx\csname urlprefix\endcsname\relax\def\urlprefix{URL }\fi

\bibitem[{Avramov et~al.(2015)Avramov, Conca, and Iyengar}]{AVARAM_LUCH_CONC}
Avramov, L.~L., Conca, A., Iyengar, S.~B., 2015. Subadditivity of syzygies of
  {K}oszul algebras. Math. Ann. 361~(1-2), 511--534.

\bibitem[{Bayer and Stillman(1988)}]{SYZ_COMP}
Bayer, D., Stillman, M., 1988. On the complexity of computing syzygies. J.
  Symbolic Comput. 6~(2-3), 135--147, computational aspects of commutative
  algebra.

\bibitem[{Bruns and Herzog(1993)}]{BRUNS_HERZ}
Bruns, W., Herzog, J., 1993. Cohen-{M}acaulay rings. Vol.~39 of Cambridge
  Studies in Advanced Mathematics. Cambridge University Press, Cambridge.

\bibitem[{Buchsbaum and Eisenbud(1977)}]{GORH3}
Buchsbaum, D.~A., Eisenbud, D., 1977. Algebra structures for finite free
  resolutions, and some structure theorems for ideals of codimension {$3$}.
  Amer. J. Math. 99~(3), 447--485.

\bibitem[{Bus\'e(2009)}]{LAURENT}
Bus\'e, L., 2009. On the equations of the moving curve ideal of a rational
  algebraic plane curve. J. Algebra 321~(8), 2317--2344.

\bibitem[{{Bus{\'e}} et~al.(2018){Bus{\'e}}, {Cid-Ruiz}, and
  {D'Andrea}}]{SAT_SPEC_FIB}
{Bus{\'e}}, L., {Cid-Ruiz}, Y., {D'Andrea}, C., 2018. {Degree and birationality
  of multi-graded rational maps}. ArXiv e-prints: \,1805.05180.

\bibitem[{Caniglia et~al.(1993)Caniglia, Corti\~nas, Dan\'on, Heintz, Krick,
  and Solern\'o}]{SERRE_ART}
Caniglia, L., Corti\~nas, G., Dan\'on, S., Heintz, J., Krick, T., Solern\'o,
  P., 1993. Algorithmic aspects of {S}uslin's proof of {S}erre's conjecture.
  Comput. Complexity 3~(1), 31--55.

\bibitem[{Chen et~al.(2005)Chen, Cox, and Liu}]{MU_BASIS}
Chen, F., Cox, D., Liu, Y., 2005. The {$\mu$} -basis and implicitization of a
  rational parametric surface. J. Symbolic Comput. 39~(6), 689 -- 706.

\bibitem[{Chen and Sederberg(2002)}]{CHEN_SEDERB_IMP}
Chen, F., Sederberg, T., 2002. A new implicit representation of a planar
  rational curve with high order singularity. Comput. Aided Geom. Design
  19~(2), 151--167.

\bibitem[{Chen and Wang(2003)}]{CHEN_WANG_REVIS}
Chen, F., Wang, W., 2003. Revisiting the {$\mu$}-basis of a rational ruled
  surface. J. Symbolic Comput. 36~(5), 699--716.

\bibitem[{Chen et~al.(2008)Chen, Wang, and Liu}]{CHEN_WANG_YANG_SING}
Chen, F., Wang, W., Liu, Y., 2008. Computing singular points of plane rational
  curves. J. Symbolic Comput. 43~(2), 92--117.

\bibitem[{Chen et~al.(2001)Chen, Zheng, and Sederberg}]{MU_BASIS_RULED_SURF}
Chen, F., Zheng, J., Sederberg, T.~W., 2001. The mu-basis of a rational ruled
  surface. Comput. Aided Geom. Design 18~(1), 61--72.

\bibitem[{Cid-Ruiz(2017)}]{DMOD}
Cid-Ruiz, Y., 2017. A ${D}$-module approach on the equations of the {R}ees
  algebra. to appear in J. Commut. Algebra, \,ArXiv:1706.06215.

\bibitem[{Cortadellas~Ben\'itez and D'Andrea(2010)}]{CARLOS_MONOID}
Cortadellas~Ben\'itez, T., D'Andrea, C., 2010. Minimal generators of the
  defining ideal of the {R}ees algebra associated to monoid parameterizations.
  Comput. Aided Geom. Design 27~(6), 461--473.

\bibitem[{Cortadellas~Ben\'itez and D'Andrea(2014)}]{CARLOS_MU2}
Cortadellas~Ben\'itez, T., D'Andrea, C., 2014. Minimal generators of the
  defining ideal of the {R}ees algebra associated with a rational plane
  parametrization with {$\mu=2$}. Canad. J. Math. 66~(6), 1225--1249.

\bibitem[{Cortadellas~Ben\'itez and D'Andrea(2015)}]{CARLOS_MONO}
Cortadellas~Ben\'itez, T., D'Andrea, C., 2015. The {R}ees algebra of a monomial
  plane parametrization. J. Symbolic Comput. 70, 71--105.

\bibitem[{Cox et~al.(2008)Cox, Hoffman, and Wang}]{COX_REES_MU_1}
Cox, D., Hoffman, J.~W., Wang, H., 2008. Syzygies and the {R}ees algebra. J.
  Pure Appl. Algebra 212~(7), 1787--1796.

\bibitem[{Cox(2001)}]{STRONG_MU_BASIS}
Cox, D.~A., 2001. Equations of parametric curves and surfaces via syzygies. In:
  Symbolic computation: solving equations in algebra, geometry, and engineering
  ({S}outh {H}adley, {MA}, 2000). Vol. 286 of Contemp. Math. Amer. Math. Soc.,
  Providence, RI, pp. 1--20.

\bibitem[{Cox(2008)}]{COX_MOVING}
Cox, D.~A., 2008. The moving curve ideal and the {R}ees algebra. Theoret.
  Comput. Sci. 392~(1-3), 23--36.

\bibitem[{Cox et~al.(2005)Cox, Little, and O'Shea}]{CLO_USING_AG}
Cox, D.~A., Little, J., O'Shea, D., 2005. Using algebraic geometry, 2nd
  Edition. Vol. 185 of Graduate Texts in Mathematics. Springer, New York.

\bibitem[{Cox et~al.(1998)Cox, Sederberg, and Chen}]{MU_BASIS_CURVES}
Cox, D.~A., Sederberg, T.~W., Chen, F., 1998. The moving line ideal basis of
  planar rational curves. Comput. Aided Geom. Design 15~(8), 803--827.

\bibitem[{Decker et~al.(2018)Decker, Greuel, Pfister, and
  Sch\"onemann}]{SINGULAR}
Decker, W., Greuel, G.-M., Pfister, G., Sch\"onemann, H., 2018. {\sc Singular}
  {4-1-1} --- {A} computer algebra system for polynomial computations.
  \url{http://www.singular.uni-kl.de}.

\bibitem[{Diesel(1996)}]{DIESEL}
Diesel, S.~J., 1996. Irreducibility and dimension theorems for families of
  height {$3$} {G}orenstein algebras. Pacific J. Math. 172~(2), 365--397.

\bibitem[{Dohm(2009)}]{DOHM}
Dohm, M., 2009. Implicitization of rational ruled surfaces with {$\mu$}-bases.
  J. Symbolic Comput. 44~(5), 479--489.

\bibitem[{Eisenbud(1995)}]{EISENBUD_COMM}
Eisenbud, D., 1995. Commutative Algebra: with a view toward algebraic geometry.
  Vol. 150 of Graduate Texts in Mathematics. Springer-Verlag, New York.

\bibitem[{Elias et~al.(2010)Elias, Giral, Mir\'o-Roig, and Zarzuela}]{SIXCOMM}
Elias, J., Giral, J.~M., Mir\'o-Roig, R.~M., Zarzuela, S. (Eds.), 2010. Six
  lectures on commutative algebra. Modern Birkh\"auser Classics. Birkh\"auser
  Verlag, Basel, papers from the Summer School on Commutative Algebra held in
  Bellaterra, July 16--26, 1996, Reprint of the 1998 edition.

\bibitem[{Hong et~al.(2008)Hong, Simis, and Vasconcelos}]{HONG_SIMIS_VASC_ELIM}
Hong, J., Simis, A., Vasconcelos, W.~V., 2008. On the homology of
  two-dimensional elimination. J. Symbolic Comput. 43~(4), 275--292.

\bibitem[{Jia and Goldman(2009)}]{GOLDMAN_JIA}
Jia, X., Goldman, R., 2009. {$\mu$}-bases and singularities of rational planar
  curves. Comput. Aided Geom. Design 26~(9), 970--988.

\bibitem[{Kaplansky(1974)}]{KAP}
Kaplansky, I., 1974. Commutative rings, revised Edition. The University of
  Chicago Press, Chicago, Ill.-London.

\bibitem[{Kustin et~al.(2011)Kustin, Polini, and Ulrich}]{KPU_NORMAL_SCROLL}
Kustin, A., Polini, C., Ulrich, B., 2011. Rational normal scrolls and the
  defining equations of {R}ees algebras. J. Reine Angew. Math. 650, 23--65.

\bibitem[{Kustin et~al.(2017)Kustin, Polini, and Ulrich}]{KPU_GOR3}
Kustin, A., Polini, C., Ulrich, B., 2017. The equations defining blowup
  algebras of height three {G}orenstein ideals. Algebra Number Theory 11~(7),
  1489--1525.

\bibitem[{Lam(2006)}]{SERRE_CONJ}
Lam, T.~Y., 2006. Serre's problem on projective modules. Springer Monographs in
  Mathematics. Springer-Verlag, Berlin.

\bibitem[{Lang(2002)}]{LANG_ALGEBRA}
Lang, S., 2002. Algebra, 3rd Edition. Vol. 211 of Graduate Texts in
  Mathematics. Springer-Verlag, New York.

\bibitem[{Lazard(1977)}]{LAZARD_ALG_LIN}
Lazard, D., 1977. Alg\`ebre lin\'eaire sur {$K[X_{1},\cdots,X_{n}]$}, et
  \'elimination. Bull. Soc. Math. France 105~(2), 165--190.

\bibitem[{Lazard(1992)}]{LAZARD_UPPER_BOUND_NOTE}
Lazard, D., 1992. A note on upper bounds for ideal-theoretic problems. J.
  Symbolic Comput. 13~(3), 231--233.

\bibitem[{Lombardi and Yengui(2005)}]{LOMBARDI}
Lombardi, H., Yengui, I., 2005. Suslin's algorithms for reduction of unimodular
  rows. J. Symbolic Comput. 39~(6), 707--717.

\bibitem[{Matsumura(1989)}]{MATSUMURA}
Matsumura, H., 1989. Commutative Ring Theory, 1st Edition. Cambridge Studies in
  Advanced Mathematics volume 8. Cambridge University Press.

\bibitem[{Migliore and Mir\'o-Roig(2003)}]{ROIG}
Migliore, J., Mir\'o-Roig, R.~M., 2003. On the minimal free resolution of
  {$n+1$} general forms. Trans. Amer. Math. Soc. 355~(1), 1--36.

\bibitem[{Migliore and Nagel(2002)}]{LIAISON}
Migliore, J., Nagel, U., may 2002. Liaison and related topics: Notes from the
  torino workshop/school. ArXiv Mathematics e-prints.

\bibitem[{Peeva(2007)}]{REGULARITY}
Peeva, I. (Ed.), 2007. Syzygies and {H}ilbert functions. Vol. 254 of Lecture
  Notes in Pure and Applied Mathematics. Chapman \& Hall/CRC, Boca Raton, FL.

\bibitem[{Peeva(2011)}]{GRADED}
Peeva, I., 2011. Graded syzygies. Vol.~14 of Algebra and Applications.
  Springer-Verlag London, Ltd., London.

\bibitem[{Peeva and Sturmfels(1998)}]{PEEVA_STURM}
Peeva, I., Sturmfels, B., 1998. Syzygies of codimension {$2$} lattice ideals.
  Math. Z. 229~(1), 163--194.

\bibitem[{Rotman(1979)}]{ROTMAN}
Rotman, J.~J., 1979. An introduction to homological algebra. Vol.~85 of Pure
  and Applied Mathematics. Academic Press, Inc.

\bibitem[{Szanto(2008)}]{SZANTO}
Szanto, A., 2008. Solving over-determined systems by the subresultant method.
  J. Symbolic Comput. 43~(1), 46--74, with an appendix by Marc Chardin.

\bibitem[{Vasconcelos(1991)}]{VASC_EQ_REES}
Vasconcelos, W.~V., 1991. On the equations of {R}ees algebras. J. Reine Angew.
  Math. 418, 189--218.

\bibitem[{Yap(1991)}]{YAP_LOWER}
Yap, C.-K., 1991. A new lower bound construction for commutative {T}hue systems
  with applications. J. Symbolic Comput. 12~(1), 1--27.

\end{thebibliography}
\bibliographystyle{elsarticle-harv}

\end{document}